\documentclass[10pt]{amsart}

\usepackage{float}
\usepackage{amscd,amsmath,amssymb,amsfonts,amsthm, ascmac, bm}
\usepackage{enumerate, comment, mathtools} 
\usepackage{lscape}

\usepackage{caption}
\usepackage{subcaption}
\usepackage{hyperref}
\usepackage{relsize}

\usepackage{booktabs}
\usepackage{a4wide}
\usepackage{verbatim}

\usepackage[mathscr]{euscript}

\usepackage{ragged2e}

\usepackage{xcolor}
\usepackage[linesnumbered,ruled,vlined]{algorithm2e}

\SetCommentSty{mycommfont}

\SetKwInput{KwInput}{Input}                
\SetKwInput{KwOutput}{Output}              

\usepackage{tikz-cd}
\usepackage{tikz, tikz-3dplot}

\numberwithin{equation}{section}
\allowdisplaybreaks[1]

\usepackage{colonequals} 

\allowdisplaybreaks

\DeclareMathOperator{\SL}{SL}

\newcommand{\C}{{\mathbb{C}}}
\newcommand{\R}{{\mathbb{R}}}

\newcommand{\Z}{{\mathbb{Z}}}
\newcommand{\Cstar}{{\C^{\ast}}}

\DeclareMathOperator{\flag}{Fl}

\theoremstyle{plain}
\newtheorem{theorem}{Theorem}[section]

\newtheorem{proposition}[theorem]{Proposition}
\newtheorem{corollary}[theorem]{Corollary}

\theoremstyle{definition}

\newtheorem{definition}[theorem]{Definition}

\theoremstyle{remark}
\newtheorem{remark}[theorem]{Remark}

\hypersetup{colorlinks}
\hypersetup{
	unicode=false,          
	pdftoolbar=true,        
	pdfmenubar=true,        
	pdffitwindow=false,     
	pdfstartview={FitH},    
	pdftitle={My title},    
	pdfauthor={Author},     
	pdfsubject={Subject},   
	pdfcreator={Creator},   
	pdfproducer={Producer}, 
	pdfkeywords={keyword1} {key2} {key3}, 
	pdfnewwindow=true,      
	colorlinks=true,       
	linkcolor=blue,          
	citecolor=red,        
	filecolor=violet,      
	urlcolor=violet       
}

\begin{document}

\author{Yunhyung Cho}
\address[Yunhyung Cho]{Department of Mathematics Education, Sungkyunkwan University, Seoul, Republic of Korea}
\email{yunhyung@skku.edu}

\author{Naoki Fujita}
\address[Naoki Fujita]{Faculty of Advanced Science and Technology, Kumamoto University, 2-39-1 Kurokami, Chuo-ku, Kumamoto 860-8555, Japan}
\email{fnaoki@kumamoto-u.ac.jp}

\author{Akihiro Higashitani}
\address[Akihiro Higashitani]{Graduate School of Information Science and Technology, Osaka University, Osaka 565-0871, Japan}
\email{higashitani@ist.osaka-u.ac.jp}

\author{Eunjeong Lee}
\address[Eunjeong Lee]{Department of Mathematics,
	Chungbuk National University,
	Cheongju 28644, Republic of Korea}
\email{eunjeong.lee@chungbuk.ac.kr}

\thanks{Lee is the corresponding author. \\
Cho was supported by the National Research Foundation of Korea(NRF) grant funded by the Korea government(MSIP; Ministry of Science, ICT \& Future Planning) (No.\ 2020R1C1C1A01010972) and (No.\ 2020R1A5A1016126). 
Fujita was supported by JSPS Grant-in-Aid for Early-Career Scientists (No.\ 20K14281) and by MEXT Japan Leading Initiative for Excellent Young Researchers (LEADER) Project. 
Lee was supported by the National Research Foundation of Korea(NRF) grant funded by the Korea government(MSIT) (No.\ RS-2023-00239947).}

\title[Newton--Okounkov polytopes of small ranks]{Newton--Okounkov polytopes of type $A$ flag varieties of small ranks arising from  cluster structures}

\date{\today}



\begin{abstract} 
A flag variety is a smooth projective homogeneous variety. 
In this paper, we study Newton--Okounkov polytopes of the flag variety $\flag(\C^4)$ arising from its cluster structure. More precisely, we present defining inequalities of such Newton--Okounkov polytopes of $\flag(\C^4)$. Moreover, we classify these polytopes, establishing their equivalence under unimodular transformations. 
\end{abstract}

\maketitle

\setcounter{tocdepth}{1} 
\tableofcontents

\section{Introduction} 
%

Let $G$ be a semisimple algebraic group, $B$ a Borel subgroup, and $H$ a maximal torus contained in $B$. The homogeneous space $G/B$ is a smooth projective variety, called a \emph{\textup{(}full\textup{)} flag variety}. 
The geometry of flag varieties provides fruitful avenues connecting the geometry, combinatorics, and representation theory (see, for example, \cite{Brion_Lecture, Fulton97Young}).  

A flag variety is not a toric variety in general but there are known toric degenerations (see, for instance, \cite{Caldero02, FFL17}). Especially, the theory of Newton--Okounkov bodies provides a method to construct toric degenerations~\cite{An13}. 
Recently, the second named author and Oya~\cite{FujitaOya} considered Newton--Okounkov polytopes of flag varieties arising from cluster structures. Moreover, they proved that the family of these polytopes contains already known interesting classes of polytopes, for instance, Berenstein--Littelmann--Zelevinsky's string polytopes~\cite{Littelmann98, BZ_tensor_product_01} (also see~\cite{Kav15}), and Nakashima--Zelevinsky polytopes~\cite{NZ_polyhedral_97, Nakashima99, FujitaSatoshi17}.
While the combinatorial description of string polytopes in type $A$ is provided in~\cite{GP00}, the description of the remaining Newton--Okounkov polytopes considered in~\cite{FujitaOya} is not so much studied.

In this paper, we consider Newton--Okounkov polytopes of~$\flag(\C^n)$ considered in~\cite{FujitaOya} for $n \leq 4$. 
Indeed, we provide their defining inequalities and study the unimodular equivalence classes. 
Here, we say that two rational convex polytopes $P$ and $Q$ in $\R^N$ are unimodularly equivalent if there exists an affine map $f \colon \mathbf{x} \mapsto A \mathbf{x} + \mathbf{b}$ with $A \in \mathrm{GL}_N(\Z)$ and $\mathbf{b} \in \Z^N$ such that $f(P) = Q$. Accordingly, toric varieties defined by unimodularly equivalent polytopes are isomorphic as algebraic varieties.

The \emph{exchange graph} $\Gamma$ for a cluster structure is a graph whose vertices are labeled by seeds and whose edges by mutations, that is, two vertices are connected by an edge if and only if the corresponding seeds are related by a mutation. 
For a cluster structure for $\flag(\C^4)$, there are $14$ seeds and the corresponding exchange graph is shown in Figure~\ref{figure_exchange_graph_A3}. We prove that there exists a certain symmetry on the exchange graph preserving the unimodularity of the corresponding Newton--Okounkov polytopes. 
Indeed, we prove that 
there exist two involutions $\iota$ and $\iota'$ on the exchange graph of type $A_3$ such that 
\[
\Delta_{\bm s} \cong \Delta_{\iota(\bm s)} \quad \text{ and } \quad 
\Delta_{\bm s} \cong \Delta_{\iota'(\bm s)}, 
\]
where $\Delta_{\bm s} = \Delta(\flag(\C^4), \mathcal{L}_{\lambda}, v_{\bm s})$ is the Newton--Okounkov body arising from the cluster structure for each seed $\bm s$ and $\mathcal{L}_{\lambda}$ is the anticanonical line bundle of $\flag(\C^4)$ (see Theorem~\ref{thm_involutions}). 
Here, for rational convex polytopes $P$ and $Q$, we denote by $P \cong Q$ if $P$ and $Q$ are unimodularly equivalent. 
Considering the $(\Z_2 \times \Z_2)$-action generated by the involutions $\iota$ and $\iota'$ on the exchange graph, we obtain 6 orbits, which implies that there exist at most 6 unimodular equivalence classes. We demonstrate that there are 5 unimodular equivalence classes, that is, two orbits provide the unimodularly equivalent polytopes. 
To obtain the above result, we use the \emph{tropicalized mutations} considered in~\cite{FujitaOya} to relate Newton--Okounkov polytopes and find defining inequalities of  $\Delta_{\bm s}$ in Section~\ref{section_NOBYs_flag4}.
%
%
%
%
%
%

This paper is organized as follows. In Section~\ref{section_NOBYs_flag4}, we consider the description of Newton--Okounkov polytopes of $\flag(\C^n)$ arising from the cluster structure for $n \leq 4$. In Section~\ref{sec_unimodular}, we provide unimodular maps among the Newton--Okounkov polytopes.

\section{Description of Newton--Okounkov polytopes}\label{section_NOBYs_flag4}
%

In this section, we recall Newton--Okounkov polytopes of flag varieties arising from cluster structures from~\cite{FujitaOya} (also, see~\cite{HaradaKaveh15, Kav15, KK12}). We notice that Newton--Okounkov bodies were introduced independently by Kaveh--Khovanskii and by Lazarsfeld--Mustata in~\cite{KK12, kaveh2008convex} and in~\cite{LM09}, respectively.
We also provide a description of Newton--Okounkov polytopes of $\flag(\C^4)$. 

We begin with the definition of Newton--Okounkov bodies.
Let $R$ be a $\C$-algebra and assume that $R$ is a domain.  Fix a total order $\leq$ on $\Z^m$, $m \in \Z_{>0}$, respecting the addition. 
\begin{definition}[{see, for instance~\cite[Definition~3.1]{HaradaKaveh15} and~\cite[Definition~2.1]{FujitaOya}}]
A map $v \colon R \setminus \{0\} \to \Z^m$ is called a \emph{valuation} on $R$ if the following holds: for every $\sigma, \tau \in R \setminus \{0\}$ and $c \in \Cstar = \C \setminus \{0\}$, 
\begin{enumerate}
\item $v(\sigma \cdot \tau) = v(\sigma) + v(\tau)$,
\item $v(c \cdot \sigma) = v(\sigma)$,
\item $v(\sigma + \tau) \geq \min \{ v(\sigma), v(\tau)\}$ unless $\sigma+\tau = 0$.
\end{enumerate}
Moreover, we say the valuation $v$ has \emph{one-dimensional leaves} if it satisfies that if $v(\sigma) = v(\tau)$, then there exists $c \in \Cstar$ such that $\tau-c\sigma = 0$ or $v(\tau-c\sigma) > v(\tau)$.
\end{definition}

A Newton--Okounkov body $\Delta(X, \mathcal{L}, v)$ is a convex body constructed from a variety $X$ with a globally generated line bundle $\mathcal{L}$ and a valuation $v$ on the function field $\C(X)$. 
\begin{definition}[{see, for instance~\cite[Section~1.2]{Kav15} and~\cite[Definition~2.4]{FujitaOya}}]
Let $X$ be an irreducible normal projective variety over $\C$ of complex dimension $m$, and let $\mathcal L$ be a line bundle on $X$ generated by global sections. Take a valuation $v \colon \C(X) \setminus \{0\} \to \Z^m$ having one-dimensional leaves and fix a nonzero section $\tau \in H^0(X, \mathcal{L})$. 
The \emph{Newton--Okounkov body} $\Delta(X,\mathcal{L},v,\tau)$ of $X$ associated with $(\mathcal{L},v,\tau)$ is defined by 
\[
\Delta(X, \mathcal{L}, v, \tau) \colonequals \overline{\bigcup_{k \in \mathbb{Z}_{>0}} \left\{\frac{1}{k} v(\sigma/\tau^k) \mid \sigma \in H^0(X, \mathcal{L}^{\otimes k}) \setminus \{0\}\right\}}.
\]
When the set $\Delta(X, \mathcal{L},v,\tau)$ is a polytope, then we call it a \emph{Newton--Okounkov polytope}. 
\end{definition}

We briefly recall (Fomin--Zelevinsky) cluster structures from~\cite{BFZ05, FZ02_I} (also, see~\cite[Section~3]{FujitaOya}). 
Let $J$ be a finite set and $J_{\text{uf}} \subset J$. 
Let $\mathcal F = \C(z_j \mid j \in J)$ be the field of rational functions in $|J|$ variables. Then a \emph{seed} ${\bm s} = ({\bm A}, \varepsilon)$ of $\mathcal F$ is a pair of 
\begin{itemize}
\item a $J$-tuple ${\bm A} = (A_j)_{j \in J}$ of elements of $\mathcal F$;
\item $\varepsilon = (\varepsilon_{i,j})_{i \in J_{\text{uf}}, j \in J} \in \text{Mat}_{J_{\text{uf}} \times J}(\Z)$
\end{itemize}
such that ${\bm A}$ forms a free generating set of $\mathcal F$, and the $J_{\text{uf}} \times J_{\text{uf}}$-submatrix of $\varepsilon$ is skew-symmetrizable, that is, there exist positive integers $(d_i)_{i \in J_{\text{uf}}}$ such that $\varepsilon_{i,j} d_i = -\varepsilon_{j,i}d_j$ for all $i,j \in J_{\text{uf}}$. We call the matrix $\varepsilon$ the \emph{exchange matrix} of ${\bm s}$. 
When an exchange matrix is skew-symmetric, then we obtain a directed graph, called a \emph{quiver}, whose adjacency matrix is the minus of the exchange matrix $\varepsilon$.

Let ${\bm s} = ({\bm A}, \varepsilon)$ be a seed of $\mathcal{F}$. For $k \in J_{\text{uf}}$, the \emph{mutation} $\mu_k({\bm s}) = (\mu_k({\bm A}), \mu_k(\varepsilon)) = ({\bm A}', \varepsilon')$ in direction $k$ is defined as follows. 
\[
\begin{split}
\varepsilon_{i,j}' &= \begin{cases} 
-\varepsilon_{i,j} & \text{ if } i = k \text{ or }j = k, \\
\varepsilon_{i,j} + \text{sgn}(\varepsilon_{i,j})[\varepsilon_{i,k} \varepsilon_{k,j}]_+ & \text{ otherwise}.
\end{cases} \\
A_i' &= \begin{cases}
\displaystyle \frac{\prod_{j \in J} A_j^{[\varepsilon_{k,j}]_+} + \prod_{j \in J} A_j^{[-\varepsilon_{k,j}]_+}}{A_k}
& \text{ if } i =k, \\
A_i & \text{ otherwise}.
\end{cases}
\end{split}
\]
Here, we write ${\bm A}^\prime = (A^\prime_j)_{j \in J}$ and $\varepsilon^\prime = (\varepsilon_{i,j}^\prime)_{i \in J_{\textrm{uf}}, j \in J}$. Moreover, we write $[a]_+ \colonequals \max\{a,0\}$.

Let $\mathbb{T}$ be the $|J_{\text{uf}}|$-regular tree whose edges are labeled by $J_{\text{uf}}$ such that the edges of  each vertex have different labels. A \emph{cluster pattern} ${\bf\mathcal{S}} = \{{\bm s}_t\}_{t \in \mathbb{T}} = \{({\bm A}_t, \varepsilon_t)\}_{t \in \mathbb{T}}$ is an assignment of a seed~${\bm s}_t$ of $\mathcal{F}$ to each vertex $t \in \mathbb{T}$ such that $\mu_k({\bm s}_t) = {\bm s}_{t'}$ whenever $\begin{tikzcd}[column sep = 1.5em]
t \arrow[r,dash, "k"] & t'
\end{tikzcd}$, that is, vertices $t$ and $t'$ are connected by an edge having label $k$. For $t \in \mathbb{T}$, we denote by $A_{t;j}$ the $j$th entry of $\bm{A}_t$ and by $\varepsilon_{i, j}^{(t)}$ an entry of $\varepsilon_t$.
We say two seeds $({\bm A}, \varepsilon)$ and $({\bm A}', \varepsilon')$ are \emph{equivalent} if there exists a permutation~$\sigma$ on $J$ such that $\sigma(j) = j$ for $j \in J \setminus J_{\textrm{uf}}$, $A_j' = A_{\sigma(j)}$, and $\varepsilon_{i,j}' = \varepsilon_{\sigma(i), \sigma(j)}$ for $i \in J_{\text{uf}}, j \in J$. 
The \emph{exchange graph} of a cluster pattern ${\bf\mathcal{S}} = \{{\bm s}_t\}_{t \in \mathbb{T}}$ is  a quotient of the tree $\mathbb{T}$ modulo the equivalence relation on the vertices defined by setting $t \sim t'$ if and only if the corresponding seeds are equivalent. 
We say that an irreducible algebraic variety $X$ has a \emph{cluster structure} if the coordinate ring $\mathbb{C}[Y]$ of an open subvariety $Y \subset X$ is isomorphic to the upper cluster algebra $\textrm{up}(\mathcal {A}) \colonequals \bigcap_{t \in \mathbb{T}} \C[A_{t;j}^{\pm 1} \mid j \in J]$ of a cluster pattern ${\bf\mathcal{S}} = \{({\bm A}_t, \varepsilon_t)\}_{t \in \mathbb{T}}$ as a $\C$-algebra.

Let $G$ be a simply-connected semisimple algebraic group over $\C$, $B$ a Borel subgroup of $G$, and $H$ a maximal torus satisfying $H \subset B$. We call the homogeneous space $G/B$ the \emph{flag variety}, which is an irreducible smooth projective variety. 
In this paper, we assume $G = \SL_n(\C)$, that is, $G$ is of type $A$. In this case, the flag variety $G/B$ is an algebraic variety of complex dimension $n(n-1)/2$ which can be identified with the set of flags:
\[
G/B \cong \flag(\C^n) \colonequals 
\{ (\{0\} \subsetneq V_1 \subsetneq \cdots \subsetneq V_{n-1} \subsetneq \C^n) \mid \dim_{\C} V_i = i \quad \text{ for all }i=1,\dots,n-1\}.
\]

The Weyl group $W \colonequals N_G(H)/H$ of $G$ is isomorphic to the symmetric group $\mathfrak{S}_n$ which is generated by simple transpositions $s_i$ for $i=1,\dots,n-1$. 
We denote by $w_0^{(n)} \in \mathfrak{S}_n$ the longest element, that is, $w_0^{(n)}(i) = n+1-i$ for $1 \leq i \leq n$, and by $\ell = \ell(w_0^{(n)})$ the length of $w_0^{(n)}$. Let
\[
R(w_0^{(n)}) \colonequals \{ {\bm i} = (i_1,\dots,i_{\ell}) \in [n-1]^{\ell} \mid w_0^{(n)} = s_{i_1}s_{i_2} \cdots s_{i_{\ell}} \}
\]
be the set of reduced words for $w_0^{(n)}$.  
For $n=4$, $\ell(w_0^{(4)})=6$ and there are $16$ reduced words as in Figure~\ref{figure_reduced_words_for_w0}. 
According to Tits' theorem~\cite{Tits69}, every pair of reduced words is connected by a sequence of the following moves:
\begin{itemize}
\item (2-move) exchanging $(i,j)$ with $(j,i)$ for $|i-j| > 1$, i.e., $s_i s_j = s_j s_i$.
\item (3-move) exchanging $(i,j,i)$ with $(j,i,j)$ for $|i-j| = 1$, i.e., $s_i s_{i+1}s_i = s_{i+1} s_i s_{i+1}$. 
\end{itemize}
In the figure, two reduced words are connected by a dashed line if they are related via a $2$-move; two reduced words are connected by a line if they are related via a $3$-move.
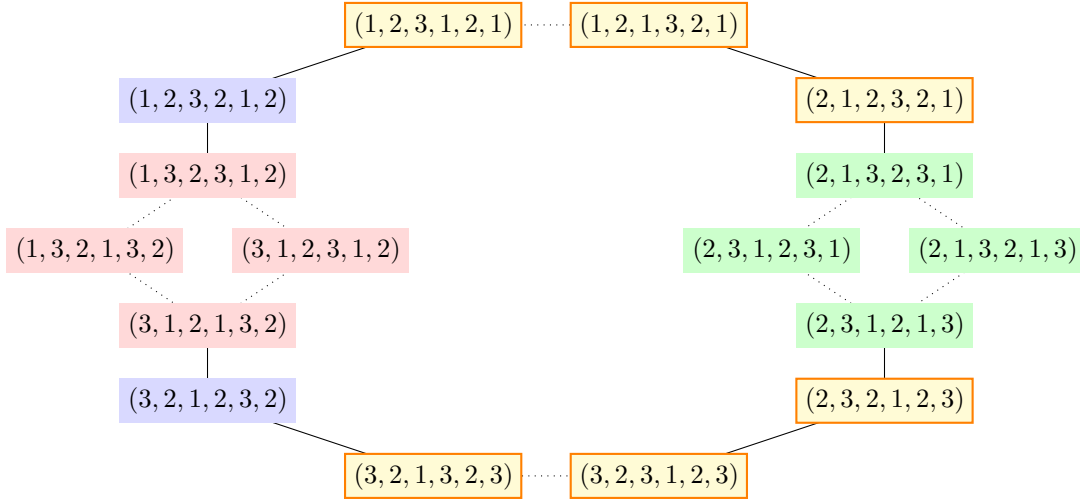
\begin{figure}
\begin{tikzpicture}[ yscale = 1, xscale = 1.5]
		

		\draw[dotted] (0,0) --(2,0);
		\draw (2,0) -- (4,1) -- (4,2);
		\draw[dotted] (4,2) -- (5,3) -- (4,4) -- (3,3) -- cycle;
		\draw (4,4) -- (4,5) -- (2,6);
		\draw[dotted] (2,6) -- (0,6);
		\draw (0,6) -- (-2,5) -- (-2,4);
		\draw[dotted] (-2,4) --(-3,3) --(-2,2) -- (-1,3) --cycle;
		\draw (-2,2) -- (-2,1) -- (0,0);
		
		\draw(0,0) node [{fill=yellow!20!white, draw=orange, thick}] {$(3,2,1,3,2,3)$};
		\draw (2,0) node [{fill=yellow!20!white, draw=orange, thick}]{$(3,2,3,1,2,3)$};
		\draw (4,1) node [fill=yellow!20!white, draw=orange, thick]{$(2,3,2,1,2,3)$};
		\draw (4,2) node [{fill=green!20!white}]{$(2,3,1,2,1,3)$};
		\draw (3,3) node [{fill=green!20!white}]{$(2,3,1,2,3,1)$};
		\draw (5,3) node [{fill=green!20!white}]{$(2,1,3,2,1,3)$};
		\draw (4,4) node [{fill=green!20!white}]{$(2,1,3,2,3,1)$};
		\draw (4,5) node [{fill=yellow!20!white, draw=orange, thick}]{$(2,1,2,3,2,1)$};
		\draw (2,6) node [{fill=yellow!20!white, draw=orange, thick}]{$(1,2,1,3,2,1)$};
		\draw (0,6) node [{fill=yellow!20!white, draw=orange, thick}]{$(1,2,3,1,2,1)$};
		\draw (-2,5) node [{fill=blue!15!white}]{$(1,2,3,2,1,2)$};
		\draw (-2,4) node [{fill=red!15!white}]{$(1,3,2,3,1,2)$};
		\draw (-1,3) node [{fill=red!15!white}]{$(3,1,2,3,1,2)$};
		\draw (-3,3) node [{fill=red!15!white}]{$(1,3,2,1,3,2)$};
		\draw (-2,2) node [{fill=red!15!white}]{$(3,1,2,1,3,2)$};
		\draw (-2,1) node [{fill=blue!15!white}]{$(3,2,1,2,3,2)$};
		\end{tikzpicture}
\caption{Reduced words for the longest element in $\mathfrak{S}_4$}\label{figure_reduced_words_for_w0}
\end{figure}

Let $B^- \subset G$ be the opposite Borel subgroup of $G$ and let $U^-$ be the unipontent radical of $B^-$. Define the unipotent cell $U^-_{w_0} \colonequals U^- \cap B \widetilde{w}_0 B$, where $\widetilde{w}_0 \in N_G(H)$ denotes a lift for the longest element $w_0 \in W$. When $G = \SL_n(\C)$, $\widetilde{w}_0^{(n)}$ is the permutation matrix representing $w_0^{(n)}$. 
Then we have the open embedding $U^- \hookrightarrow G/B$ and it induces an open embedding $U^-_{w_0} \hookrightarrow G/B$. 

The cluster structure on the double Bruhat cell $G^{w_0,e} \coloneqq B^- \cap B \widetilde{w}_0 B$ introduced in~\cite{BFZ05}
induces that on the unipotent cell $U^-_{w_0}$ (see, for instance,~\cite[Appendix~B]{FujitaOya}).
Let us consider a cluster pattern $\mathcal{S} = \{{\bm s}_t\}_{t \in \mathbb{T}}$ of $\mathbb{C}[U^-_{w_0}]$.
The second named author and Oya constructed a valuation~$v_{{\bm s}_t}$ for each seed ${\bm s}_t$, and moreover, they proved that the Newton--Okounkov polytopes of $G/B$ arising from the cluster pattern are related by \emph{tropicalized mutations}. We also notice that each reduced word ${\bm i}$ for the longest element associates to a seed ${\bm s}_{\bm i}$ of the cluster pattern ${\bf\mathcal{S}}$ (see~\cite{BFZ05}). 
\begin{theorem}[{\cite[Corollaries 5.7(2), 6.6(1), 6.30 and Theorem 6.8(1)]{FujitaOya}}]\label{thm_Fujita_Oya}
Let $\lambda$ be a dominant integral weight, $\mathcal{L}_\lambda$ the corresponding line bundle on $G/B$, and $t \in \mathbb{T}$. Then the following statements hold for a specific section $\tau_\lambda \in H^0(G/B, \mathcal{L}_\lambda)$.
\begin{enumerate}
\item The Newton--Okounkov body $\Delta(G/B, \mathcal{L}_{\lambda}, v_{\bm{s}_{\bm i}}, \tau_{\lambda})$ is unimodularly equivalent to the string polytope $\Delta_{\bm i}(\lambda)$.
\item The Newton--Okounkov body $\Delta(G/B, \mathcal{L}_{\lambda}, v_{\bm{s}_{t}}, \tau_{\lambda})$ is a rational convex polytope. 
\item If $\begin{tikzcd}[column sep = 1.5em]
t \arrow[r,dash, "k"] & t'
\end{tikzcd}$, then 
\[
\Delta(G/B, \mathcal{L}_{\lambda}, v_{{\bm s}_{t'}}, \tau_{\lambda}) 
= \mu_k^T(\Delta(G/B, \mathcal{L}_{\lambda}, v_{{\bm s}_t}, \tau_{\lambda})). 
\]
Here, $\mu_k^T \colon \R^{J} \to \R^J, (g_j)_{j \in J} \mapsto (g_j')_{j \in J}$ is defined by
\[
g_j' = \begin{cases}
g_j + [-\varepsilon_{k,j}^{(t)}]_+ g_k + \varepsilon_{k,j}^{(t)}[g_k]_+ & \text{ if } j \neq k, \\
-g_j &\text{ if }j = k
\end{cases}
\]
for $j \in J$. 
\item The string polytopes $\Delta_{\bm i}(\lambda)$ and the Nakashima--Zelevinsky polytopes $\widetilde{\Delta}_{\bm i}(\lambda)$ are related by tropicalized mutations up to unimodular transformations.
\end{enumerate}
\end{theorem}

When $G = \SL_3(\C)$, the cluster structure on the unipotent cell has $2$ seeds. Indeed, both come from reduced words of the longest element $w_0^{(3)}$. We depict two quivers below.
\[
\begin{array}{cc}
\begin{tikzcd}[row sep = 0]
& 2 \\
1 \arrow[ur] && 3 \arrow[ll] 
\end{tikzcd} 
& \begin{tikzcd}[row sep = 0]
& 2 \arrow[ld] \\
1 \arrow[rr] && 3 
\end{tikzcd} 
\\
{\bm i}_0 = (1,2,1) & {\bm i}_1 = (2,1,2)
\end{array}
\]
They produce the string polytopes unimodularly equivalent to the Gelfand--Cetlin polytope (see, for instance, \cite[Theorem~A]{CKLP21}). The facet normal vectors of the polytope given by the seed ${\bm s}_{{\bm i}_0}$ are given by the column vectors of the following matrix.
\[
\begin{bmatrix}
1 & 0 & 0 & -1 & 0 & 0 \\
1 & 1 & 0 & 0 & -1 & 0 \\
0 & 0 & 1 & -1 & 0 & -1
\end{bmatrix}.
\]

When $G = \SL_4(\C)$, the cluster structure on the unipotent cell has $14$ seeds. Moreover, in this case, the exchange graph can be realized as the one-skeleton of a polytope, called a \emph{generalized associahedron} (see~\cite{CFZ02}). 
We depict the generalized associahedron in Figure~\ref{figure_associahedron_A3}. 
\begin{figure}
	\tdplotsetmaincoords{110}{-30}
	\begin{tikzpicture}%
	[tdplot_main_coords,
	scale= 2,
	scale=0.700000,
	back/.style={loosely dotted, thin},
	edge/.style={color=black, thick},
	facet/.style={fill=blue!95!black,fill opacity=0.100000},
	vertex/.style={inner sep=1pt,circle,fill=black,thick,anchor=base},
	gvertex/.style={inner sep=1.2pt,circle,draw=green!25!black,fill=green!75!black,thick,anchor=base}]
	\coordinate (1) at (-0.50000, -1.50000, 2.00000);
	\coordinate (2) at (1.50000, 1.50000, -2.00000);
	\coordinate (3) at (0.50000, 0.50000, 1.00000);
	\coordinate (4) at (0.50000, 1.50000, 0.00000);
	\coordinate (5) at (1.50000, 1.50000, -1.00000);
	\coordinate (6) at (-0.50000, -0.50000, 2.00000);
	\coordinate (7) at (1.50000, 0.50000, 0.00000);
	\coordinate (8) at (1.50000, -1.50000, 0.00000);
	\coordinate (9) at (1.50000, -1.50000, -2.00000);
	\coordinate (10) at (-1.50000, -1.50000, -2.00000);
	\coordinate (11) at (-1.50000, -1.50000, 2.00000);
	\coordinate (12) at (-1.50000, 1.50000, -2.00000);
	\coordinate (13) at (-1.50000, 1.50000, 0.00000);
	\coordinate (14) at (-1.50000, -0.50000, 2.00000);
	
	
	\draw[edge,back] (9) -- (10);
	\draw[edge,back] (10) -- (11);
	\draw[edge,back] (10) -- (12);
	\node[vertex] at (10)     {};
	
%
	
	\draw[edge] (1) -- (6);
	\draw[edge] (1) -- (8);
	\draw[edge] (1) -- (11);
	\draw[edge] (2) -- (5);
	\draw[edge] (2) -- (9);
	\draw[edge] (2) -- (12);
	\draw[edge] (3) -- (4);
	\draw[edge] (3) -- (6);
	\draw[edge] (3) -- (7);
	\draw[edge] (4) -- (5);
	\draw[edge] (4) -- (13);
	\draw[edge] (5) -- (7);
	\draw[edge] (6) -- (14);
	\draw[edge] (7) -- (8);
	\draw[edge] (8) -- (9);
	\draw[edge] (11) -- (14);
	\draw[edge] (12) -- (13);
	\draw[edge] (13) -- (14);
	\node[vertex] at (1)     {};
	\node[vertex] at (2)     {};
	\node[vertex] at (3)     {};
	\node[vertex] at (4)     {};
	\node[vertex] at (5)     {};
	\node[vertex] at (6)     {};
	\node[vertex] at (7)     {};
	\node[vertex] at (8)     {};
	\node[vertex] at (9)     {};
	\node[vertex] at (11)     {};
	\node[vertex] at (12)     {};
	\node[vertex] at (13)     {};
	\node[vertex] at (14)     {};
	

\draw[blue, very thick] (14)--(6)--(1)--(8)--(9)--(2)--(12)--(13)--(14);


	\foreach \y/\z in {13/0, 12/8, 9/12, 8/13}{
	\node[fill=yellow!20, draw = black, thick, minimum size= 0.5cm, inner sep = 0pt, circle] at (\y) {$\z$};
}

	\foreach \y in {2}{
	\node 	[fill=green!20, draw=black, thick,  minimum size= 0.5cm, inner sep = 0pt, circle] at (\y) {$11$};
}

\foreach \y in {6}{
\node [fill=red!15,  draw=black, thick, minimum size= 0.5cm, inner sep = 0pt, circle] at (\y) {$9$};
}

\foreach \y/\z in {14/7, 1/10}{
\node [fill=blue!15,  draw=black, thick,  minimum size= 0.5cm, inner sep = 0pt, circle] at (\y) {$\z$};
}

\foreach \y/\z in {3/2, 4/1, 5/4, 7/3, 10/6, 11/5}{
\node [fill=white,  draw=black, thick,  minimum size= 0.5cm, inner sep = 0pt, circle] at (\y) {$\z$};
 }

	\end{tikzpicture}
\caption{The one-skeleton of the generalized associahedron is the exchange graph. }\label{figure_associahedron_A3}
\end{figure}
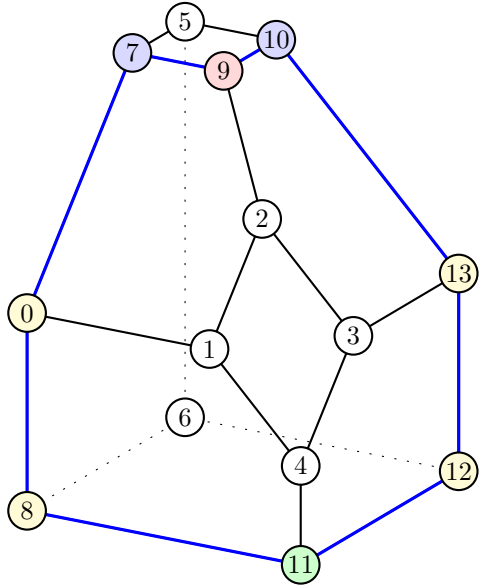
Moreover, we draw the quivers in the diagram on page~\pageref{land_scape_page}. 
Among $14$ seeds, only $8$ seeds come from the reduced words of the longest element. Those seeds are labeled with numbers $0,7,8,\dots,13$. The corresponding reduced words are listed in the first and second columns of Table~\ref{table_seed_reduced_words}. Here, we note that two reduced words produce the same seed if they are related via a sequence of $2$-moves and we write in the table only one reduced word among reduced words related to sequences of $2$-moves.
\begin{remark}
By Theorem~\ref{thm_Fujita_Oya}(4), the string polytopes $\Delta_{\bm i}(\lambda)$ and the Nakashima--Zelevinsky polytopes $\widetilde{\Delta}_{\bm i}(\lambda)$ are related by tropicalized  mutations. For each reduced word $\bm i$ of the longest element $w_0^{(4)}$, applying the mutation sequence $\overleftarrow{\bm \mu}_{\bm i}$ given in~\cite[Section 6.2.1]{FujitaOya}, we associate each Nakashima--Zelevinsky polytope $\widetilde{\Delta}_{\bm i}(\lambda)$ to a seed ${\bm s}_{t}$ such that the Newton--Okounkov polytope $\Delta(G/B, \mathcal{L}_{\lambda}, v_{{\bm s}_{t}}, \tau_{\lambda})$ is unimodularly equivalent to $\widetilde{\Delta}_{\bm i}(\lambda)$ in the third column of Table~\ref{table_seed_reduced_words}. 
\end{remark}
\begin{table}[b]
\begin{tabular}{c|| c | c}
\toprule 
reduced words & seed ${\bm s}_t  = {\bm s}_{\bm i}$ & seed related to NZ polytopes\\
\midrule 
$(1,2,1,3,2,1)$ & $0$ & $13$\\
$(1,2,3,2,1,2)$ & $7$ &  $3$\\
$(1,3,2,3,1,2)$  & $9$ & $4$\\
$(2,1,2,3,2,1)$ & $8$ & $10$\\
$(2,1,3,2,3,1)$ & $11$ & $5$ \\
$(2,3,2,1,2,3)$ & $12$ & $7$\\
$(3,2,1,2,3,2)$ & $10$ & $1$\\
$(3,2,1,3,2,3)$ & $13$ & $0$\\
\bottomrule 
\end{tabular}
\caption{Seeds come from reduced words and related to Nakashima--Zelevinsky polytopes. 
The second column represents the seeds come from reduced words, that is, ${\bm s}_{t} = {\bm s}_{\bm i}$. The third column represents the seeds related to Nakashima--Zelevinsky polytopes, that is,  $\widetilde{\Delta}_{\bm i}(\lambda) \cong \Delta(G/B, \mathcal{L}_{\lambda}, v_{{\bm s}_{t}}, \tau_{\lambda})$.
}\label{table_seed_reduced_words}
\end{table}

Since there is a systematic way to describe the string polytopes using the combinatorics of reduced words \cite{GP00}, we describe the remaining six Newton--Okounkov polytopes $\Delta(G/B, \mathcal{L}_{\lambda}, v_{{\bm s}_t}, \tau_{\lambda})$ for $t = t_1,\dots,t_6$ using the tropicalized mutations described in Theorem~\ref{thm_Fujita_Oya}(3). Here, $t_i$ denotes an element of $\mathbb{T}$ giving the seed labeled with the number $i$.
For simplicity, we set $v_{t} = v_{\bm s_{t}}$. 
Throughout this paper, we set $\lambda = 2 \varpi_1 + \cdots + 2 \varpi_{n-1}$, which defines the anticanonical line bundle $\mathcal{L}_{\lambda}$ of $\flag(\C^n)$. Here,  $\varpi_1,\dots,\varpi_{n-1}$ are fundamental weights. 
It is proved in~\cite[Theorem~3.8]{AB04} that if $\Delta = \Delta(\flag(\C^4), \mathcal{L}_{\lambda}, v_{t}, \tau_{\lambda})$ is a lattice polytope, then it is reflexive. Since the Newton--Okounkov bodies that we considered are all lattice polytopes, each polytope~$\Delta$ can be represented in the following form (after an appropriate translation): 
\[
\Delta = \{ \mathbf g \in \R^6 \mid \langle \mathbf g, \mathbf u_F \rangle \geq -1 \quad \text{ for all facets }F \}.
\]
We present a matrix whose column vectors are the facet normal vectors $\mathbf{u}_F$ for each case. 

We start with the seed $\mathbf{s}_{t_0}$ coming from the reduced word $\bm i_0 = (1,2,1,3,2,1)$. The corresponding quiver corresponding to $\varepsilon^{{\bm i}_0}$ is given as follows:
\begin{equation}\label{eq_quiver_0}
\begin{tikzcd}[row sep = 0]
&& 4 \\
& 2 \arrow[ur] \arrow[dr] && 5 \arrow[ll] \\
1 \arrow[ur] && 3 \arrow[ll] \arrow[ur] && 6 \arrow[ll]
\end{tikzcd}
\end{equation}
Here, $J = \{1,\dots,6\}$ and $J_{\text{uf}} = \{1,2,3\}$. 
In this case, the facet normal vectors of the polytope $\Delta(\flag(\C^4), \mathcal{L}_{\lambda}, v_{t_0}, \tau_{\lambda})$ are the column vectors of the following matrix. 
\begin{equation}\label{eq_normal_vectors_121321}
\begin{bmatrix}
1 & 0 & 0 & 0 & 0 & 0 & -1 & 0 & 0 & 0 & 0 & 0 \\
1 & 1 & 0 & 0 & 0 & 0 & 0 & -1 & 0 & 0 & 0 & 0 \\
0 & 0 & 1 & 0 & 0 & 0 & -1 & 0 & -1 & 0 & 0 & 0 \\
1 & 1 & 0 & 1 & 0 & 0 & 0 & 0 & 0 & -1 & 0 & 0 \\
0 & 0 & 1 & 0 & 1 & 0 & 0 & -1 & 0 & 0 & -1 & 0 \\
0 & 0 & 0 & 0 & 0 & 1 & -1 & 0 & -1 & 0 & 0 & -1
\end{bmatrix}. 
\end{equation}

To obtain the above vectors, we consider the relation between the string polytope $\Delta_{{\bm i}_0}(\lambda)$ and the Newton--Okounkov polytope $\Delta(\flag(\C^4), \mathcal{L}_{\lambda}, v_{t_0}, \tau_{\lambda})$. Using the description in~\cite[Section~5]{Littelmann98} (or  Gleizer--Postnikov's rigorous paths in~\cite{GP00}), for $\lambda = \lambda_1 \varpi_1 + \lambda_2 \varpi_2 + \lambda_3 \varpi_3$, the string polytope~$\Delta_{{\bm i}_0}(\lambda)$ is given by the set of points $(a_1,\dots,a_6) \in \mathbb{R}^6_{\geq 0}$ satisfying the following inequalities:
\[
\begin{array}{rcccl}
	0 & \leq & a_1 & \leq & a_2 - 2a_3 + a_5 - 2a_6 + \lambda_1, \\
	a_3 & \leq & a_2 & \leq & a_3 + a_4 - 2a_5 + a_6 + \lambda_2, \\
	0 & \leq & a_3 & \leq & a_5 - 2a_6 + \lambda_1, \\
	a_5 & \leq & a_4 & \leq & a_5 + \lambda_3, \\
	a_6 & \leq & a_5 & \leq & a_6 + \lambda_2, \\
	0 & \leq & a_6 & \leq & \lambda_1.
\end{array}
\]
On the other hand, by~\cite[Corollary~6.6]{FujitaOya} (also, see~\cite[Theorem~2.14(3)]{FujitaOya}), we obtain $\Delta(\flag(\C^4), \mathcal{L}_{\lambda}, v_{t_0}, \tau_{\lambda}) = \Delta_{{\bm i}_0}(\lambda) M_{{\bm i}_0}$, where $M_{{\bm i}_0}$ is given by 
\[
M_{{\bm i}_0} = 
\begin{bmatrix}
	1 & 0 & 0 & 0 & 0 & 0\\
	1 & 1 & 0 & 0 & 0 & 0\\
	0 & 1 & 1 & 0 & 0 & 0\\
	1 & 1 & 0 & 1 & 0 & 0 \\
	0 & 1 & 1 & 1 & 1 & 0 \\
	0 & 0 & 0 & 1 & 1 & 1
\end{bmatrix}.
\]
Accordingly, by computing 
\[
\begin{bmatrix}
	1 & 0 & 0 & 0 & 0 & 0\\
	1 & 1 & 0 & 0 & 0 & 0\\
	0 & 1 & 1 & 0 & 0 & 0\\
	1 & 1 & 0 & 1 & 0 & 0 \\
	0 & 1 & 1 & 1 & 1 & 0 \\
	0 & 0 & 0 & 1 & 1 & 1
\end{bmatrix}
\begin{bmatrix}
	1 & 0 & 0 & 0 & 0 & 0 & -1 & 0 & 0 & 0 & 0 & 0 \\
	0 & 1 & 0 & 0 & 0 & 0 & 1 & -1 & 0 & 0 & 0 & 0 \\
	0 & -1 & 1 & 0 & 0 & 0 & -2 & 1 & -1 & 0 & 0 & 0 \\
	0 & 0 & 0 & 1 & 0 & 0 & 0 & 1 & 0 & -1 & 0 & 0 \\
	0 & 0 & 0 & -1 & 1 & 0 & 1 & -2 & 1 & 1 & -1 & 0 \\
	0 & 0 & 0 & 0 & -1 & 1 & -2 & 1 & -2 & 0 & 1 & -1
\end{bmatrix},
\]
we obtain~\eqref{eq_normal_vectors_121321}.

For $t=t_1,\dots,t_6$, we compute the defining inequalities of the Newton--Okounkov body $\Delta(\flag(\C^4), \mathcal{L}_{\lambda}, v_{t}, \tau_{\lambda})$ applying Theorem~\ref{thm_Fujita_Oya}(3) repeatedly using the computer program SAGE. See Algorithm~\ref{algorithm}. 
We provide more explanation on this. To obtain the image under the map $\mu_k^T$ of a polytope $P$, 
we first divide $P$ into two pieces via the coordinate hyperplane $g_k = 0$. Indeed, we obtain a polytope $P_{+}$ lives in $g_k \geq 0$; a polytope $P_{-}$ lives in $g_k \leq 0$ such that $P = P_{+} \cup P_-$. 
The tropicalized mutation $\mu_k^T$ becomes a linear map when $g_k \geq 0$ or $g_k \leq 0$. We denote two linear maps by $T_+$ and $T_-$, respectively. To obtain the images $\mu_k^T(P_+)$ and $\mu_k^T(P_-)$, it is enough to consider images $T_+(V(P_+))$ and $T_-(V(P_-))$ of the vertices of $P_+$ and $P_-$ under linear maps $T_+$ and $T_-$. 
By taking the convex hull of $T_+(V(P_+)) \cup T_-(V(P_-))$, we obtain $\mu_k^T(P)$. 
 
\begin{algorithm}[!ht]
\DontPrintSemicolon
  
  \KwInput{Polytope $P$, coordinate index $k$ of the hyperplane $g_{k}=0$, two transformations $T_+$ and $T_-$ defining the tropicalized mutation}
  \KwOutput{Polytope after applying the tropicalized mutation}

  $P_{+} =$ polytope whose defining inequalities are those of $P$ and $g_{k} \geq 0$;

  $P_{-} =$ polytope whose defining inequalities are those of $P$ and $g_{k} \leq 0$; 

  $V = \emptyset$;
  \tcc{V will be used to produce the new polytope. }

  \For{each vertex $v$ of $P_{+}$}
    {
        $V \leftarrow V \cup T_+(v)$; 
    }
  \For{each vertex $v$ of $P_{-}$}
	{
		$V \leftarrow V \cup T_-(v)$;
	}

  \Return{Polytope which is the convex hull of $V$};
\caption{Applying tropicalized mutation}\label{algorithm}
\end{algorithm}

\medskip 
\textsf{Case 1. $t = t_1$. }
We note that 
$\begin{tikzcd}[column sep = 1.5em]
t_0 \arrow[r,dash, "2"] & t_1
\end{tikzcd}$.
Applying Theorem~\ref{thm_Fujita_Oya}(3) to the quiver~\eqref{eq_quiver_0}, we obtain the map $\mu_2^T$ as follows:
\[
\mu_2^T(g_1,\dots,g_6) = (g_1 + [g_2]_+, -g_2, g_3 - [-g_2]_+, g_4 - [-g_2]_+, g_5 +[g_2]_+, g_6).
\]
Using Algorithm~\ref{algorithm}, we obtain the facet normal vectors of the polytope $\Delta(G/B, \mathcal{L}_{\lambda}, v_{t_1}, \tau_{\lambda})$ as the column vectors of the following matrix. 
\[
\begin{bmatrix}
1 &  0 &  0 &  0 &  0 &  0 &  -1 &  0 &  0 &  0 &  0 &  0 &  0 \\
0 &  1 &  1 &  0 &  0 &  0 &  -1 &  -1 &  -1 &  0 &  0 &  0 &  0\\
0 &  0 &  1 &  0 &  0 &  0 &  -1 &  -1 &  0 &  -1 &  0 &  0 &  0\\
1 &  0 &  0 &  1 &  0 &  0 &  0 &  0 &  -1 &  0 &  -1 &  0 &  0\\
0 &  1 &  1 &  0 &  1 &  0 &  0 &  0 &  0 &  0 &  0 &  -1 &  0\\
0 &  0 &  0 &  0 &  0 &  1 &  -1 &  -1 &  0 &  -1 &  0 &  0 &  -1
\end{bmatrix}.
\]
We note that the quiver corresponding to $t_1$ is given as follows:
\begin{equation}\label{eq_quiver_1}
\begin{tikzcd}[row sep = 0]
&& 4 \arrow[dl] \\
& 2 \arrow[dl] \arrow[rr] && 5 \\
1 \arrow[rruu, bend left] && 3 \arrow[ul] && 6 \arrow[ll]
\end{tikzcd}
\end{equation}

\medskip 

\textsf{Case 2. $t = t_2$. } We note that $\begin{tikzcd}[column sep = 1.5em]
t_1 \arrow[r,dash, "3"] & t_2
\end{tikzcd}$.
Applying Theorem~\ref{thm_Fujita_Oya}(3) to the quiver~\eqref{eq_quiver_1}, we obtain the map $\mu_3^T$ as follows:
\[
\mu_3^T(g_1,\dots,g_6) = 
(g_1, g_2 -[-g_3]_+, -g_3, g_4, g_5, [g_3]_+ + g_6).
\]
Applying Algorithm~\ref{algorithm}, we obtain the facet normal vectors of the polytope $\Delta(G/B, \mathcal{L}_{\lambda}, v_{t_2}, \tau_{\lambda})$ as the column vectors of the following matrix. 
\[
\begin{bmatrix}
1 &  0 &  0 &  0 &  0 &  0 &  -1 &  0 &  0 &  0 &  0 &  0 &  0\\
0 &  1 &  0 &  0 &  0 &  0 &  -1 &  -1 &  -1 &  -1 &  0 &  0 &  0\\
0 &  0 &  1 &  0 &  0 &  0 &  0 &  -1 &  0 &  0 &  0 &  0 &  0\\
1 &  0 &  0 &  1 &  0 &  0 &  0 &  -1 &  0 &  -1 &  -1 &  0 &  0\\
0 &  1 &  0 &  0 &  1 &  0 &  0 &  0 &  0 &  0 &  0 &  -1 &  0\\
0 &  0 &  1 &  0 &  0 &  1 &  -1 &  0 &  -1 &  0 &  0 &  0 &  -1
\end{bmatrix}.
\]
We note that the quiver corresponding to $t_2$ is given as follows:
\begin{equation}\label{eq_quiver_2}
\begin{tikzcd}[row sep = 0]
&& 4 \arrow[dl] \\
& 2 \arrow[dl] \arrow[rr]\arrow[rd] && 5 \\
1 \arrow[rruu, bend left] && 3 \arrow[rr] && 6 \arrow[lllu]
\end{tikzcd}
\end{equation}

\medskip 
\textsf{Case 3. $t = t_3$. }
We note that $\begin{tikzcd}[column sep = 1.5em]
t_2 \arrow[r,dash, "1"] & t_3
\end{tikzcd}$.
Applying Theorem~\ref{thm_Fujita_Oya}(3) to the quiver~\eqref{eq_quiver_2}, we obtain the map $\mu_1^T$ as follows:
\[
\mu_1^T(g_1,\dots,g_6) = 
(-g_1, [g_1]_+ + g_2, g_3, - [-g_1]_+ + g_4, g_5, g_6).
\]
Applying Algorithm~\ref{algorithm}, we obtain the facet normal vectors of the polytope $\Delta(G/B, \mathcal{L}_{\lambda}, v_{t_3}, \tau_{\lambda})$ as the column vectors of the following matrix.  
\[
\begin{bmatrix}
1 &  0 &  0 &  0 &  0 &  0 &  -1 &  -1 &  -1 &  0 &  0 &  0 &  0\\
1 &  1 &  0 &  0 &  0 &  0 &  0 &  -1 &  -1 &  -1 &  0 &  0 &  0\\
0 &  0 &  1 &  0 &  0 &  0 &  0 &  0 &  -1 &  0 &  0 &  0 &  0\\
0 &  0 &  0 &  1 &  0 &  0 &  -1 &  -1 &  -1 &  0 &  -1 &  0 &  0\\
1 &  1 &  0 &  0 &  1 &  0 &  0 &  0 &  0 &  0 &  0 &  -1 &  0\\
0 &  0 &  1 &  0 &  0 &  1 &  0 &  0 &  0 &  -1 &  0 &  0 &  -1
\end{bmatrix}.
\]

\medskip 
\textsf{Case 4. $t = t_4$. }
We note that $\begin{tikzcd}[column sep = 1.5em]
t_1 \arrow[r,dash, "1"] & t_4
\end{tikzcd}$.
Applying Theorem~\ref{thm_Fujita_Oya}(3) to the quiver~\eqref{eq_quiver_1}, we obtain the map $\mu_1^T$ as follows:
\[
\mu_1^T(g_1,\dots,g_6) = 
(-g_1, [g_1]_+ + g_2, g_3, -[-g_1]_+ + g_4, g_5, g_6).
\]
Applying Algorithm~\ref{algorithm}, we obtain the facet normal vectors of the polytope $\Delta(G/B, \mathcal{L}_{\lambda}, v_{t_4}, \tau_{\lambda})$ as the column vectors of the following matrix. 

\[
\begin{bmatrix}
1 &  1 &  0 &  0 &  0 &  0 &  0 &  -1 &  -1 &  0 &  0 &  0 &  0 &  0\\
1 &  1 &  1 &  1 &  0 &  0 &  0 &  0 &  -1 &  -1 &  0 &  0 &  0 &  0\\
1 &  0 &  1 &  0 &  0 &  0 &  0 &  0 &  0 &  -1 &  -1 &  0 &  0 &  0\\
0 &  0 &  0 &  0 &  1 &  0 &  0 &  -1 &  -1 &  0 &  0 &  -1 &  0 &  0\\
1 &  1 &  1 &  1 &  0 &  1 &  0 &  0 &  0 &  0 &  0 &  0 &  -1 &  0\\
0 &  0 &  0 &  0 &  0 &  0 &  1 &  0 &  0 &  -1 &  -1 &  0 &  0 &  -1
\end{bmatrix}.
\]

\medskip 

\textsf{Case 5. $t = t_5$. }
We note that $\begin{tikzcd}[column sep = 1.5em]
t_0 \arrow[r,dash, "3"] & t_7 \arrow[r, dash, "1"] & t_5
\end{tikzcd}$ and the quiver corresponding to $t_7$ is given by 
\begin{equation}\label{eq_quiver_7}
\begin{tikzcd}[row sep = 0]
&& 4 \\
& 2 \arrow[ur]   && 5 \arrow[ld] \\
1 \arrow[rr] && 3 \arrow[ul] \arrow[rr]  && 6 \arrow[llll, bend left]
\end{tikzcd}
\end{equation}
Applying Theorem~\ref{thm_Fujita_Oya}(3) to the quivers~\eqref{eq_quiver_0} and~\eqref{eq_quiver_7}, we obtain the facet normal vectors of the polytope $\Delta(G/B, \mathcal{L}_{\lambda}, v_{t_5}, \tau_{\lambda})$ as the column vectors of the following matrix. 
\[
\begin{bmatrix}
1 &  1 &  0 &  0 &  0 &  0 &  0 &  -1 &  -1 &  0 &  0 &  0 &  0 &  0\\
0 &  0 &  1 &  1 &  0 &  0 &  0 &  -1 &  0 &  -1 &  0 &  0 &  0 &  0\\
0 &  1 &  1 &  0 &  0 &  0 &  0 &  -1 &  -1 &  -1 &  -1 &  0 &  0 &  0\\
0 &  0 &  1 &  1 &  1 &  0 &  0 &  0 &  0 &  0 &  0 &  -1 &  0 &  0\\
0 &  0 &  0 &  0 &  0 &  1 &  0 &  -1 &  -1 &  -1 &  -1 &  0 &  -1 &  0\\
1 &  1 &  0 &  0 &  0 &  0 &  1 &  0 &  0 &  0 &  0 &  0 &  0 &  -1
\end{bmatrix}.
\]

\medskip 
\textsf{Case 6. $t = t_6$. } 
We note that $\begin{tikzcd}[column sep = 1.5em]
t_0 \arrow[r,dash, "1"] & t_8 \arrow[r, dash, "3"] & t_6
\end{tikzcd}$ and the quiver corresponding to $t_8$ is given by
\begin{equation}\label{eq_quiver_8}
\begin{tikzcd}[row sep = 0]
&& 4   \\
& 2 \arrow[dl] \arrow[ur]  && 5 \arrow[ll] \\
1 \arrow[rr] && 3 \arrow[ru] && 6 \arrow[ll]
\end{tikzcd}
\end{equation}
Applying Theorem~\ref{thm_Fujita_Oya}(3) to the quivers~\eqref{eq_quiver_0} and~\eqref{eq_quiver_8}, we obtain the facet normal vectors of the polytope $\Delta(G/B, \mathcal{L}_{\lambda}, v_{t_2}, \tau_{\lambda})$ as the column vectors of the following matrix. 

\[
\begin{bmatrix}
1 &  0 &  0 &  0 &  0 &  0 &  -1 &  0 &  0 &  0 &  0 &  0 &  0\\
0 &  1 &  0 &  0 &  0 &  0 &  -1 &  -1 &  -1 &  0 &  0 &  0 &  0\\
0 &  0 &  1 &  0 &  0 &  0 &  -1 &  0 &  -1 &  -1 &  0 &  0 &  0\\
0 &  1 &  0 &  1 &  0 &  0 &  0 &  0 &  0 &  0 &  -1 &  0 &  0\\
1 &  0 &  0 &  0 &  1 &  0 &  -1 &  -1 &  -1 &  -1 &  0 &  -1 &  0\\
0 &  0 &  1 &  0 &  0 &  1 &  0 &  0 &  0 &  0 &  0 &  0 &  -1
\end{bmatrix}.
\]

\begin{remark}
We notice that the Newton--Okounkov polytopes  $\Delta(\flag(\C^4), \mathcal{L}_{\lambda},v_t, \tau_{\lambda})$ are all lattice polytopes.
\end{remark}
\section{Unimodular equivalence classes}\label{sec_unimodular}
In this section, we consider the unimodular equivalence classes among Newton--Okounkov polytopes $\Delta(\flag(\C^4), \mathcal{L}_{\lambda},v_t, \tau_{\lambda})$ arising from the cluster structure. 

\begin{proposition}\label{prop_unimodular}
Among $14$ seeds in the cluster structure for $\flag(\C^4)$, there are $5$ polytopes up to unimodular equivalence as in Table~\ref{table_equivalence_f_vector}.
\end{proposition}

In Table~\ref{table_equivalence_f_vector}, we provide the seeds providing the polytopes in the same equivalence classes and their $f$-vectors. 
\begin{table}[b]
\begin{tabular}{c|c|c}
\toprule 
Case No. & seeds & $f$-vectors \\
\midrule 
Case 1 & $0, 8, 12, 13$ & $(1, 40, 132, 186, 139, 57, 12, 1)$ \\
Case 2 & $1,3,7,10$ & $(1, 42, 141, 202, 153, 63, 13, 1)$ \\
Case 3 & $2,9$ & $(1, 38, 133, 197, 152, 63, 13, 1)$ \\
Case 4 & $4,5$ & $(1, 43, 146, 212, 163, 68, 14, 1)$ \\
Case 5 & $6,11$ & $(1, 42, 141, 202, 153, 63, 13, 1)$ \\
\bottomrule 
\end{tabular}
\caption{Equivalence classes and $f$-vectors}\label{table_equivalence_f_vector}
\end{table} 
We notice that the polytopes in Cases 2 and 5 have the same $f$-vector. However, they produce combinatorially different polytopes. Considering the degree, which is number of edges emanating from each vertex, we obtain Table~\ref{table_degrees_Case25} and we can see that the numbers of vertices having a given degree are different. 
\begin{table}
\begin{tabular}{c|rr}
	\toprule 
	degree &  Case 2 &  Case 5 \\
	\midrule 
	6 & 20 & 22\\
	7 & 15 & 10\\
	8 & 4 & 8\\
	9 & 3 & 2 \\
	\bottomrule 
\end{tabular}
\caption{Numbers of vertices of polytopes in Cases 2 and 5  having a given degree}
\label{table_degrees_Case25}
\end{table}

Using Proposition~\ref{prop_unimodular}, we obtain the following observation, which is the main theorem of this paper. 
\begin{theorem}\label{thm_involutions}
There exist two involutions $\iota$ and $\iota'$ on the exchange graph for $\flag(\mathbb{C}^4)$ such that 
\[
\Delta_{\bm s} \cong \Delta_{\iota(\bm s)} \quad \text{ and } \quad 
\Delta_{\bm s} \cong \Delta_{\iota'(\bm s)}. 
\]
Here, $\Delta_{\bm s} = \Delta(\flag(\C^4), \mathcal{L}_{\lambda}, v_{\bm s}, \tau_{\lambda})$ is the Newton--Okounkov body arising from the cluster structure for each seed $\bm s$ and $\mathcal{L}_{\lambda}$ is the anticanonical line bundle of $\flag(\C^4)$. 
Moreover, there is no other symmetries on the exchange graph preserving the unimodular equivalence classes.  
\end{theorem}
\begin{proof} 
By analyzing the classification provided in Proposition~\ref{prop_unimodular}, we obtain Figure~\ref{figure_exchange_graph_A3} having $(\Z_2 \times \Z_2)$-symmetry. More precisely, there exist two involutions $\iota$ and $\iota'$ on the exchange graph for $\flag(\C^4)$ such that 
\[
\Delta_{\bm s} \cong \Delta_{\iota(\bm s)} \quad \text{ and } \quad 
\Delta_{\bm s} \cong \Delta_{\iota'(\bm s)}. 
\]
This proves the statement. 
\end{proof}
\begin{figure}[b]
\begin{tikzpicture}[scale=1.3]
\coordinate (0) at (-1.5,0);
\coordinate (1) at (-0.5, -1);
\coordinate (2) at (0,-0.5); 
\coordinate (3) at (0.5,-1); 
\coordinate (4) at (0,-1.5);
\coordinate (5) at (0,1.5);
\coordinate (6) at (0,2.5);
\coordinate (7) at (-0.5,1);
\coordinate (8) at (-2.5,0);
\coordinate (9) at (0,0.5);
\coordinate (10) at (0.5,1);
\coordinate (11) at (0,-2.5);
\coordinate (12) at (2.5,0);
\coordinate (13) at (1.5,0); 

\foreach \x in {0,...,13}{
	\node at (\x) {$\x$};	
}

\draw[thick, red, dashed, text =black] (-3.5,0)--(3.5,0) node[at end] {$\updownarrow \iota'$};
\draw[thick, red, dashed, text=black] (0,-3.5)--(0,3.5) node[at end] {$\leftrightarrow$}; 
\node at (0.3,3.5) {$\iota$};

\draw (0)--(7)--(9)--(2)--(1)--(0)
(0)--(8)--(6)--(5)--(7)
(5)--(10)--(9)
(10)--(13)--(3)--(2)
(1)--(4)--(11)--(8)
(3)--(4)
(11)--(12)--(6)
(12)--(13); 

	\foreach \x in {0, 8, 12, 13}{
	\node[fill=yellow!20, circle, draw = black, inner sep = 0cm, minimum size= 0.65cm] at (\x) {$\x$};
}

	\foreach \x in {11,6}{
	\node[fill=green!20, circle, draw = black, inner sep = 0cm, minimum size= 0.65cm] at (\x) {$\x$};
}

	\foreach \x in {9,2}{
	\node[fill=red!15, circle, draw = black, inner sep = 0cm, minimum size= 0.65cm] at (\x) {$\x$};
}

	\foreach \x in {7,10,1,3}{
	\node[fill=blue!15, circle, draw = black, inner sep = 0cm, minimum size= 0.65cm] at (\x) {$\x$};
}

	\foreach \x in {4,5}{
	\node[fill=orange!15, circle, draw = black, inner sep = 0cm, minimum size= 0.65cm] at (\x) {$\x$};
}
\end{tikzpicture}
\caption{The exchange graph  for $\flag(\C^4)$}\label{figure_exchange_graph_A3}
\end{figure}

\begin{remark}
The involution $\iota$ corresponds to the involution on the Dynkin diagram of type~$A_3$ which interchanges $s_1$ and $s_3$. For instance, the seed ${\bm s}_{t_0}$ corresponds to the reduced word $(1,2,1,3,2,1)$, and the involution on the Dynkin diagram sends it to $(3,2,3,1,2,3)$. The reduced word $(3,2,3,1,2,3)$ corresponds to the seed ${\bm s}_{t_{13}}$. 
We note that the involution $\iota$ is obviously extended to general type $A_{n-1}$ using the automorphism of ${\rm SL}_n(\mathbb{C})$ induced from the involution of the Dynkin diagram.  
\end{remark}
Before providing unimodular maps among the polytopes, we notice that the polytope in Case $1$ has $12$ facets and the other polytopes have more facets. Moreover, the polytope $\Delta_{{\bm s}_{t_0}}$ is unimodularly equivalent to the Gelfand--Cetlin polytope. Accordingly, we have the following observation. 
\begin{corollary}\label{cor_GC_type}
Let $\Delta_{\bm s_t} = \Delta(\flag(\C^4), \mathcal{L}_{\lambda}, v_t, \tau_{\lambda})$ be the Newton--Okounkov body arising from the cluster structure. Then the following are equivalent. 
\begin{enumerate}
\item The polytope $\Delta_{\bm s_t}$ is unimodularly equivalent to the Gelfand--Cetlin polytope. 
\item The polytope $\Delta_{\bm s_t}$ has exactly $12$ facets. 
\end{enumerate}
Moreover, if one of the above holds, then the seed ${\bm s}_t$ comes from a reduced word of the longest element. 
\end{corollary}
One can think of that Corollary~\ref{cor_GC_type} generalizes the result~\cite[Theorem~A]{CKLP21}, which says that the Gelfand--Cetlin type string polytopes are characterized by the number of facets. Indeed, the above corollary shows that the number of facets also can be used to characterize the Newton--Okounkov polytopes $\Delta_{\bm s_t}$ which are unimodularly equivalent to the Gelfand--Cetlin polytope.

\subsection{Proof of Proposition~\ref{prop_unimodular}}\label{subsection_unimodular}
To find unimodular equivalence classes, we first consider \emph{combinatorial equivalence} classes using the SAGE. Indeed, the SAGE command \texttt{is\_combinatorially\_isomorphic()} is used to determine whether two polytopes are combinatorially equivalent or not. After then, we find an explicit unimodular map among the polytopes having the same combinatorial structure. 

\medskip 

\textsf{Case 1. } There are four seeds with numbers $0,8,12,13$ producing the combinatorially same polytopes. Note that these four seeds come from reduced words of longest elements. See Table~\ref{table_seed_reduced_words}. 
Moreover, it is known from~\cite[Theorem~A]{CKLP21} that they produce the string polytopes unimodularly equivalent to the Gelfand--Cetlin polytope. Indeed, the seed ${\bm s}_{t_0}$ is obtained from the reduced word ${\bm i}_0 = (1,2,1,3,2,1)$ and the facet normal vectors of the polytope given by the seed ${\bm s}_{t_0}$ are the column vectors of the matrix in~\eqref{eq_normal_vectors_121321}. 

\medskip 

\textsf{Case 2. } 
There are four seeds with numbers $1,3,7,10$ producing the combinatorially same polytopes. Seeds $t_7$ and $t_{10}$ come from the reduced words, indeed, $t_7$ corresponds to $(1,2,3,2,1,2)$ and $t_{10}$ corresponds to $(3,2,1,2,3,2)$. Since the involution on the Dynkin diagram provides 
a unimodular map among string polytopes, we obtain 
\[
\Delta_{{\bm s}_{t_7}} \cong \Delta_{{\bm s}_{t_{10}}}.
\]
On the other hand, considering the coordinate change $(g_1,g_2,g_3,g_4,g_5,g_6)$ $\mapsto (g_3, g_2, g_1, g_6, g_5, g_4)$, we obtain 
\[
\Delta_{{\bm s}_{t_1}} \cong \Delta_{{\bm s}_{t_{3}}}.
\]
To describe a unimodular map between polytopes $\Delta_{{\bm s}_{t_1}}$ and $\Delta_{{\bm s}_{t_7}}$, we consider the facet normal vectors of the polytope $\Delta_{{\bm s}_{t_7}}$. Note that the polytope $\Delta_{{\bm s}_{t_7}}$ is unimodularly equivalent to the string polytope and we get the following facet normal vectors using Gleizer--Postnikov's rigorous paths in~\cite{GP00} (also, see~\cite[Section~4]{CKLP21}). 
We present a matrix whose column vectors are the facet normal vectors. 
\[
\begin{bmatrix}
1 & 0 & 0 & 0 & 0 & 0 & 0 & -1 & 0 & 0 & 0 & 0 & 0 \\
1 & 1 & 1 & 0 & 0 & 0 & 0 & 0 & -1 & 0 & 0 & 0 & 0 \\
1 & 1 & 1 & 1 & 0 & 0 & 0 & 0 & 0 & -1 & 0 & 0 & 0 \\
1 & 1 & 0 & 0 & 1 & 0 & 0 & 0 & -1 & 0 & -1 & 0 & 0 \\
0 & 0 & 0 & 0 & 1 & 1 & 0 & -1 & 0 & 0 & 0 & -1 & 0 \\
0 & 0 & 0 & 0 & 0 & 0 & 1 & 0 & -1 & 0 & 0 & 0 & -1
\end{bmatrix}. 
\]
One can see that the coordinate change $(g_1,g_2,g_3,g_4,g_5,g_6)$ $\mapsto (-g_1, -g_3, -g_6, -g_2, -g_4, -g_5)$ induces the unimoular map from $\Delta_{{\bm s}_{t_7}}$ to $\Delta_{{\bm s}_{t_1}}$. 

\medskip 

\textsf{Case 3. }
There are two seeds with numbers $2,9$ producing the combinatorially same polytopes. Note that the seed ${\bm s}_{t_9}$ corresponds to the reduced word $(1,3,2,3,1,2)$. Using Gleizer--Postnikov's rigorous paths in~\cite{GP00}, we obtain the following matrix whose column vectors are the facet normal vectors of the polytope $\Delta_{{\bm s}_{t_9}}$.
\[
\begin{bmatrix}
1 & 0 & 0 & 0 & 0 & 0 & 0 & -1 & 0 & 0 & 0 & 0 & 0 \\
0 & 1 & 0 & 0 & 0 & 0 & 0 & 0 & -1 & 0 & 0 & 0 & 0 \\
1 & 1 & 1 & 1 & 0 & 0 & 0 & 0 & 0 & -1 & 0 & 0 & 0 \\
1 & 0 & 1 & 0 & 1 & 0 & 0 & 0 & -1 & 0 & -1 & 0 & 0 \\
0 & 1 & 0 & 1 & 0 & 1 & 0 & -1 & 0 & 0 & 0 & -1 & 0 \\
0 & 0 & 0 & 0 & 0 & 0 & 1 & 0 & 0 & -1 & 0 & 0 & -1
\end{bmatrix}.
\]
One can see that the coordinate change $(g_1,g_2,g_3,g_4,g_5,g_6)$ $\mapsto (-g_3, -g_1, -g_2, -g_4, -g_6, -g_5)$ induces a unimodular map from $\Delta_{{\bm s}_{t_9}}$ to $\Delta_{{\bm s}_{t_2}}$. 

\medskip 

\textsf{Case 4. }
There are two seeds with numbers $4,5$ producing the combinatorially same polytopes. 
One can see that the coordinate change $(g_1,g_2,g_3,g_4,g_5,g_6)$ $\mapsto (-g_2, -g_3, -g_1, -g_4, -g_5, -g_6)$ induces a unimodular map from $\Delta_{{\bm s}_{t_4}}$ to $\Delta_{{\bm s}_{t_5}}$.

\medskip 

\textsf{Case 5. }
There are two seeds with numbers $6,11$ producing the combinatorially same polytopes. 
Note that the seed ${\bm s}_{t_{11}}$ corresponds to the reduced word $(2,1,3,2,3,1)$. Using Gleizer--Postnikov's rigorous paths in~\cite{GP00}, we obtain the following matrix whose column vectors are the facet normal vectors of the polytope $\Delta_{{\bm s}_{t_{11}}}$.
\[
\begin{bmatrix}
1 & 0 & 0 & 0 & 0 & 0 & 0 & -1 & 0 & 0 & 0 & 0 & 0 \\
1 & 1 & 1 & 0 & 0 & 0 & 0 & 0 & -1 & 0 & 0 & 0 & 0 \\
1 & 1 & 0 & 1 & 0 & 0 & 0 & 0 & 0 & -1 & 0 & 0 & 0 \\
1 & 1 & 1 & 1 & 1 & 0 & 0 & -1 & 0 & 0 & -1 & 0 & 0 \\
0 & 0 & 0 & 0 & 0 & 1 & 0 & 0 & 0 & -1 & 0 & -1 & 0 \\
0 & 0 & 0 & 0 & 0 & 0 & 1 & 0 & -1 & 0 & 0 & 0 & -1
\end{bmatrix}.
\]
One can see that the coordinate change $(g_1,g_2,g_3,g_4,g_5,g_6)$ $\mapsto (-g_1, -g_3, -g_2, -g_5, -g_4, -g_6)$ induces a unimodular map from $\Delta_{{\bm s}_{t_{11}}}$ to $\Delta_{{\bm s}_{t_6}}$.

\begin{remark}
We notice that the unimodular transformations provided in this section do not come from the compositions of tropicalized mutations considered in Section~\ref{section_NOBYs_flag4}. Indeed, the composition of tropicalized mutations does \emph{not} need to be an affine transformation. 
\end{remark}

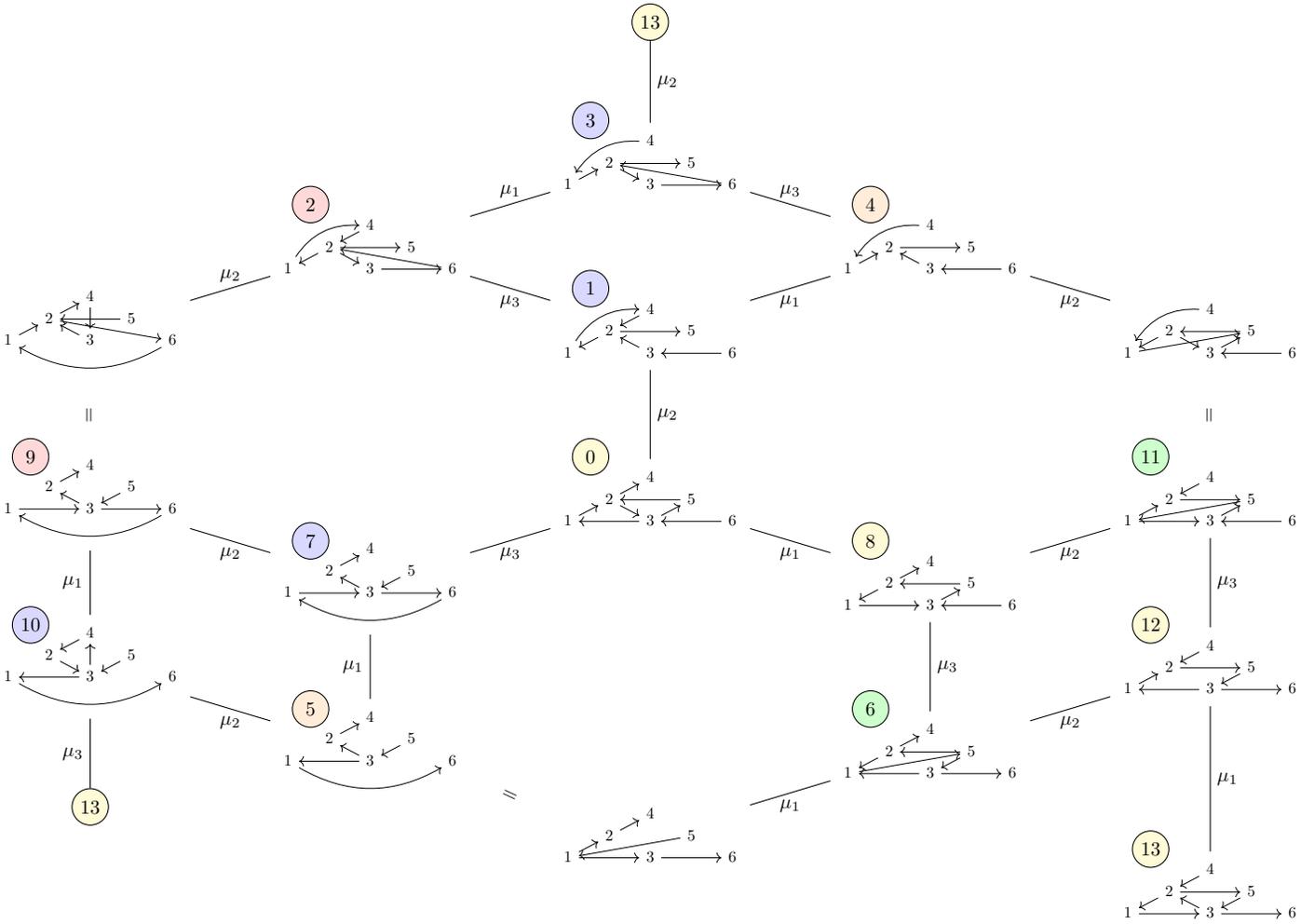
\begin{landscape}
\begin{tikzpicture}[xscale = 0.8, yscale = 0.8, every node/.style={scale=0.8}]

\node (t0) at (0,0) {
$\begin{tikzcd}[row sep = 0, column sep = 1em]
&& 4 \\
& 2 \arrow[ur] \arrow[dr] && 5 \arrow[ll] \\
1 \arrow[ur] && 3 \arrow[ll] \arrow[ur] && 6 \arrow[ll]
\end{tikzcd}$};
\node[fill=yellow!20, circle, draw = black, inner sep = 0cm, minimum size= 0.65cm] at ([shift={(145:1.3)}]t0) {$0$};

\node(t1) at (0,3) {
$\begin{tikzcd}[row sep = 0, column sep = 1em]
&& 4 \arrow[dl] \\
& 2 \arrow[dl] \arrow[rr] && 5 \\
1 \arrow[rruu, bend left] && 3 \arrow[ul] && 6 \arrow[ll]
\end{tikzcd}$};
\node[ fill=blue!15, circle, draw = black, inner sep = 0cm, minimum size= 0.65cm] at ([shift={(145:1.3)}]t1) {$1$};

\node (t2) at (-5,4.5) {
$\begin{tikzcd}[row sep = 0, column sep = 1em]
&& 4 \arrow[dl] \\
& 2 \arrow[dl] \arrow[rr]\arrow[rd] && 5 \\
1 \arrow[rruu, bend left] && 3 \arrow[rr] && 6 \arrow[lllu]
\end{tikzcd}$
};
\node[ fill=red!15, circle, draw = black, inner sep = 0cm, minimum size= 0.65cm] at ([shift={(145:1.3)}]t2) {$2$};

\node (t3) at (0, 6){
$\begin{tikzcd}[row sep = 0, column sep = 1em]
&& 4  \arrow [lldd, bend right] \\
& 2  \arrow[rr]\arrow[rd] && 5 \\
1 \arrow[ur]  && 3 \arrow[rr] && 6 \arrow[lllu]
\end{tikzcd}$
};
\node[ fill=blue!15, circle, draw = black, inner sep = 0cm, minimum size= 0.65cm] at ([shift={(145:1.3)}]t3) {$3$};

\node (t4) at (5,4.5){
$\begin{tikzcd}[row sep = 0, column sep = 1em]
&& 4  \arrow [lldd, bend right] \\
& 2  \arrow[rr]  && 5 \\
1 \arrow[ur]  && 3 \arrow[ul] && 6  \arrow[ll]
\end{tikzcd}$
};
\node[fill=orange!15, circle, draw = black, inner sep = 0cm, minimum size= 0.65cm] at ([shift={(145:1.3)}]t4) {$4$};

\node (t7) at (-5, -1.5){
$\begin{tikzcd}[row sep = 0, column sep = 1em]
&& 4  \\
& 2 \arrow[ur] && 5 \arrow[ld] \\
1 \arrow[rr] && 3 \arrow[ul] \arrow[rr] && 6 \arrow[llll, bend left]
\end{tikzcd}$
};
\node[fill=blue!15, circle, draw = black, inner sep = 0cm, minimum size= 0.65cm] at ([shift={(145:1.3)}]t7) {$7$};

\node (9a) at (-10, 0){
$\begin{tikzcd}[row sep = 0, column sep = 1em]
&& 4  \\
& 2 \arrow[ur] && 5 \arrow[ld] \\
1 \arrow[rr] && 3 \arrow[ul] \arrow[rr] && 6 \arrow[llll, bend left]
\end{tikzcd}$
};
\node[fill=red!15, circle, draw = black, inner sep = 0cm, minimum size= 0.65cm] at ([shift={(145:1.3)}]9a) {$9$};

\node (9b) at (-10, 3){
$\begin{tikzcd}[row sep = 0, column sep = 1em]
&& 4 \arrow[dd] \\
& 2  \arrow[rrrd] \arrow[ru]  && 5 \arrow[ll] \\
1 \arrow[ru] && 3 \arrow[ul] && 6 \arrow[llll, bend left]
\end{tikzcd}$
};
\node at ($(9a)!0.5!(9b)$) {\rotatebox{90}{$=$}};

\node (t10) at (-10, -3){
$\begin{tikzcd}[row sep = 0, column sep = 1em]
&& 4 \arrow[dl] \\
& 2  \arrow[dr] && 5 \arrow[dl] \\
1 \arrow[rrrr, bend right] && 3 \arrow[ll] \arrow[uu] && 6 
\end{tikzcd}$
};
\node[fill=blue!15, circle, draw = black, inner sep = 0cm, minimum size= 0.65cm] at ([shift={(145:1.3)}]t10) {$10$};

\node (5a) at (-5, -4.5){
$\begin{tikzcd}[row sep = 0, column sep = 1em]
&& 4  \\
& 2  \arrow[ur] && 5 \arrow[dl] \\
1 \arrow[rrrr, bend right] && 3 \arrow[ll] \arrow[ul] && 6 
\end{tikzcd}$
};
\node[fill=orange!15, circle, draw = black, inner sep = 0cm, minimum size= 0.65cm] at ([shift={(145:1.3)}]5a) {$5$};

\node (5b) at (0,-6){
$\begin{tikzcd}[row sep = 0, column sep = 1em]
&& 4  \\
& 2  \arrow[ur] && 5 \arrow[dlll] \\
1 \arrow[ur] \arrow[rr] && 3 \arrow[rr] && 6 
\end{tikzcd}$
};

\node at ($(5a)!0.5!(5b)$) {\rotatebox{155}{$=$}};

\node (t8) at (5,-1.5){
$\begin{tikzcd}[row sep = 0, column sep = 1em]
&& 4  \\
& 2  \arrow[ur] \arrow[dl] && 5 \arrow[ll] \\
1 \arrow[rr] && 3 \arrow[ur] && 6 \arrow[ll]
\end{tikzcd}$
};
\node[fill=yellow!20, circle, draw = black, inner sep = 0cm, minimum size= 0.65cm] at ([shift={(145:1.3)}]t8) {$8$};

\node (t6) at (5,-4.5){
$\begin{tikzcd}[row sep = 0, column sep = 1em]
&& 4  \\
& 2  \arrow[ur] \arrow[dl] && 5 \arrow[ll] \arrow [ld] \\
1 \arrow[rrru] && 3 \arrow[ll] \arrow[rr] && 6 
\end{tikzcd}$
};
\node[fill=green!20, circle, draw = black, inner sep = 0cm, minimum size= 0.65cm] at ([shift={(145:1.3)}]t6) {$6$};

\node (t12) at (10, -3){
$\begin{tikzcd}[row sep = 0, column sep = 1em]
&& 4 \arrow[ld] \\
& 2  \arrow[rr] && 5 \arrow [ld] \\
1 \arrow[ru] && 3 \arrow[ll] \arrow[rr] && 6 
\end{tikzcd}$
};
\node[fill=yellow!20, circle, draw = black, inner sep = 0cm, minimum size= 0.65cm] at ([shift={(145:1.3)}]t12) {$12$};

\node[fill=yellow!20, circle, draw = black, inner sep = 0cm, minimum size= 0.65cm] at ([shift={(90:2.5)}]t3) (13a) {$13$}; 
\node[fill=yellow!20, circle, draw = black, inner sep = 0cm, minimum size= 0.65cm] at ([shift={(-90:2.5)}]t10) (13b) {$13$};

\node (t13) at (10,-7) {
$\begin{tikzcd}[row sep = 0, column sep = 1em]
&& 4 \arrow[ld] \\
& 2  \arrow[rr] \arrow[ld] && 5 \arrow [ld] \\
1 \arrow[rr] && 3 \arrow[lu] \arrow[rr] && 6 
\end{tikzcd}$
};
\node[fill=yellow!20, circle, draw = black, inner sep = 0cm, minimum size= 0.65cm] at ([shift={(145:1.3)}]t13) {$13$};

\node (11a) at (10, 0){
$\begin{tikzcd}[row sep = 0, column sep = 1em]
&& 4 \arrow[ld] \\
& 2  \arrow[rr] && 5 \arrow [llld] \\
1 \arrow[ru] \arrow[rr] && 3 \arrow[ru]  && 6 \arrow[ll]
\end{tikzcd}$
};
\node[fill=green!20, circle, draw = black, inner sep = 0cm, minimum size= 0.65cm] at ([shift={(145:1.3)}]11a) {$11$};

\node (11b) at(10, 3){
$\begin{tikzcd}[row sep = 0, column sep = 1em]
&& 4 \arrow[lldd, bend right] \\
& 2  \arrow[ld] \arrow[rd] && 5 \arrow [ll] \\
1 \arrow[rrru] && 3 \arrow[ru]  && 6 \arrow[ll]
\end{tikzcd}$
};
\node at ($(11a)!0.5!(11b)$) {\rotatebox{90}{$=$}};

\draw (t0)--(t1) node[midway, right] {$\mu_2$};
\draw (t1)--(t2) node[midway, below] {$\mu_3$};
\draw (t2)--(t3) node[midway, above] {$\mu_1$};
\draw (t3)--(t4) node[midway, above] {$\mu_3$};
\draw (t4)--(t1) node[midway, below] {$\mu_1$};
\draw (t3)--(13a) node[midway, right] {$\mu_2$};
\draw (t0)--(t7) node[midway, below] {$\mu_3$};
\draw (t7)--(9a) node[midway, below] {$\mu_2$};
\draw (t2)--(9b) node[midway, above] {$\mu_2$};
\draw (9a)--(t10) node[midway, left] {$\mu_1$};
\draw (t10)--(5a) node[midway, below] {$\mu_2$};
\draw (5a)--(t7) node[midway, left] {$\mu_1$};
\draw (t0)--(t8) node[midway, below] {$\mu_1$};
\draw (t8)--(t6) node[midway, right] {$\mu_3$};
\draw (t6)--(5b) node[midway, below] {$\mu_1$};
\draw (t6)--(t12) node[midway, below] {$\mu_2$};
\draw (t12)--(t13) node[midway, right] {$\mu_1$};
\draw (t10)--(13b) node[midway, left] {$\mu_3$};
\draw (t12)--(11a) node[midway, right] {$\mu_3$};
\draw (t8)--(11a) node[midway, below] {$\mu_2$};
\draw (11b)--(t4) node[midway, below] {$\mu_2$};
\end{tikzpicture}

\label{land_scape_page}
\end{landscape}


\begin{thebibliography}{10}
	
	\bibitem{AB04}
	Valery Alexeev and Michel Brion.
	\newblock Toric degenerations of spherical varieties.
	\newblock {\em Selecta Math. (N.S.)}, 10(4):453--478, 2004.
	
	\bibitem{An13}
	Dave Anderson.
	\newblock Okounkov bodies and toric degenerations.
	\newblock {\em Math. Ann.}, 356(3):1183--1202, 2013.
	
	\bibitem{BFZ05}
	Arkady Berenstein, Sergey Fomin, and Andrei Zelevinsky.
	\newblock Cluster algebras. {III}. {U}pper bounds and double {B}ruhat cells.
	\newblock {\em Duke Math. J.}, 126(1):1--52, 2005.
	
	\bibitem{BZ_tensor_product_01}
	Arkady Berenstein and Andrei Zelevinsky.
	\newblock Tensor product multiplicities, canonical bases and totally positive
	varieties.
	\newblock {\em Invent. Math.}, 143(1):77--128, 2001.
	
	\bibitem{Brion_Lecture}
	Michel Brion.
	\newblock Lectures on the geometry of flag varieties.
	\newblock In {\em Topics in cohomological studies of algebraic varieties},
	Trends Math., pages 33--85. Birkh\"{a}user, Basel, 2005.
	
	\bibitem{Caldero02}
	Philippe Caldero.
	\newblock Toric degenerations of {S}chubert varieties.
	\newblock {\em Transform. Groups}, 7(1):51--60, 2002.
	
	\bibitem{CFZ02}
	Fr\'{e}d\'{e}ric Chapoton, Sergey Fomin, and Andrei Zelevinsky.
	\newblock Polytopal realizations of generalized associahedra.
	\newblock volume~45, pages 537--566. 2002.
	\newblock Dedicated to Robert V. Moody.
	
	\bibitem{CKLP21}
	Yunhyung Cho, Yoosik Kim, Eunjeong Lee, and Kyeong-Dong Park.
	\newblock On the combinatorics of string polytopes.
	\newblock {\em J. Combin. Theory Ser. A}, 184:Paper No. 105508, 46, 2021.
	
	\bibitem{FFL17}
	Xin Fang, Ghislain Fourier, and Peter Littelmann.
	\newblock On toric degenerations of flag varieties.
	\newblock In {\em Representation theory---current trends and perspectives}, EMS
	Ser. Congr. Rep., pages 187--232. Eur. Math. Soc., Z\"{u}rich, 2017.
	
	\bibitem{FZ02_I}
	Sergey Fomin and Andrei Zelevinsky.
	\newblock Cluster algebras. {I}. {F}oundations.
	\newblock {\em J. Amer. Math. Soc.}, 15(2):497--529, 2002.
	
	\bibitem{FujitaSatoshi17}
	Naoki Fujita and Satoshi Naito.
	\newblock Newton-{O}kounkov convex bodies of {S}chubert varieties and
	polyhedral realizations of crystal bases.
	\newblock {\em Math. Z.}, 285(1-2):325--352, 2017.
	
	\bibitem{FujitaOya}
	Naoki Fujita and Hironori Oya.
	\newblock {N}ewton--{O}kounkov polytopes of {S}chubert varieties arising from
	cluster structures.
	\newblock arXiv:2002.09912v2, preprint 2023.
	
	\bibitem{Fulton97Young}
	William Fulton.
	\newblock {\em Young tableaux}, volume~35 of {\em London Mathematical Society
		Student Texts}.
	\newblock Cambridge University Press, Cambridge, 1997.
	\newblock With applications to representation theory and geometry.
	
	\bibitem{GP00}
	Oleg Gleizer and Alexander Postnikov.
	\newblock Littlewood-{R}ichardson coefficients via {Y}ang-{B}axter equation.
	\newblock {\em Internat. Math. Res. Notices}, (14):741--774, 2000.
	
	\bibitem{HaradaKaveh15}
	Megumi Harada and Kiumars Kaveh.
	\newblock Integrable systems, toric degenerations and {O}kounkov bodies.
	\newblock {\em Invent. Math.}, 202(3):927--985, 2015.
	
	\bibitem{Kav15}
	Kiumars Kaveh.
	\newblock Crystal bases and {N}ewton-{O}kounkov bodies.
	\newblock {\em Duke Math. J.}, 164(13):2461--2506, 2015.
	
	\bibitem{KK12}
	Kiumars Kaveh and A.~G. Khovanskii.
	\newblock Newton--{O}kounkov bodies, semigroups of integral points, graded
	algebras and intersection theory.
	\newblock {\em Ann. of Math. (2)}, 176(2):925--978, 2012.
	
	\bibitem{kaveh2008convex}
	Kiumars Kaveh and Askold~G. Khovanskii.
	\newblock Convex bodies and algebraic equations on affine varieties.
	\newblock arXiv:0804.4095v1; a short version with title Algebraic equations and
	convex bodies appeared in Perspectives in Analysis, Geometry, and Topology,
	Progr. Math. Vol. 296, Birkh\"{a}user/Springer, New York, 2012, 263--282,
	preprint 2018.
	
	\bibitem{LM09}
	Robert Lazarsfeld and Mircea Musta\c{t}\u{a}.
	\newblock Convex bodies associated to linear series.
	\newblock {\em Ann. Sci. \'{E}c. Norm. Sup\'{e}r. (4)}, 42(5):783--835, 2009.
	
	\bibitem{Littelmann98}
	Peter Littelmann.
	\newblock Cones, crystals, and patterns.
	\newblock {\em Transform. Groups}, 3(2):145--179, 1998.
	
	\bibitem{Nakashima99}
	Toshiki Nakashima.
	\newblock Polyhedral realizations of crystal bases for integrable highest
	weight modules.
	\newblock {\em J. Algebra}, 219(2):571--597, 1999.
	
	\bibitem{NZ_polyhedral_97}
	Toshiki Nakashima and Andrei Zelevinsky.
	\newblock Polyhedral realizations of crystal bases for quantized {K}ac--{M}oody
	algebras.
	\newblock {\em Adv. Math.}, 131(1):253--278, 1997.
	
	\bibitem{Tits69}
	Jacques Tits.
	\newblock Le probl\`eme des mots dans les groupes de {C}oxeter.
	\newblock In {\em Symposia {M}athematica ({INDAM}, {R}ome, 1967/68), {V}ol. 1},
	pages 175--185. Academic Press, London-New York, 1969.
	
\end{thebibliography}

\end{document}